\documentclass[12pt]{article}
\usepackage{graphicx}
\usepackage{latexsym}
\usepackage{epsfig}
\usepackage{psfrag}
\usepackage{amssymb}
\usepackage{times}
\usepackage{subfigure}
\usepackage{amsfonts}
\usepackage{mathrsfs}
\usepackage{enumerate}

\newtheorem {algo}{Algorithm}
\newtheorem {thm} {Theorem}[section] 
\newtheorem {lem}[thm]{Lemma}               
\newtheorem {cor}[thm]{Corollary}            
\newcommand{\qed}{\hfill $\Box$}
\newcommand{\ignore}[1]{}

\newenvironment{proof}{{\bf Proof.}}{\hfill\qed\par\bigskip}
\newcommand{\rmenum}

\date{}
\begin{document}
\title{A note on Matching Cover Algorithm
\thanks{Research supported by NSFC (11571323) and NSF-Henan (15IRTSTHN006).}}

\author{ {Xiumei Wang$^1$\thanks{Corresponding author: Xiumei Wang. e-mail: wangxiumei@zzu.edu.cn.}, \ Xiaoxin Song$^{2}$, \ Jinjiang Yuan$^1$
}\\
{\small $^{1}$ School of Mathematics and Statistics, Zhengzhou University,}\\
{\small Zhengzhou 450001, People's Republic of  China}\\ {\small
$^{2}$ College of Mathematics and Information Science, Henan
University,}\\
{\small  Kaifeng 475001, People's Republic of  China}}

\maketitle

\begin{abstract}

A {\it $k$-matching cover} of a graph $G$ is a union of $k$
matchings of $G$ which covers $V(G)$.  A matching cover of $G$ is {\it
optimal} if it consists of the fewest matchings of $G$. In this paper,
we present an algorithm for finding an optimal matching cover of a
graph  on $n$ vertices and $m$ edges in $O(nm)$ time. This algorithm corrects an error of Matching Cover
Algorithm in (Xiumei Wang,  Xiaoxin Song, Jinjiang Yuan, On matching cover of graphs, Math. Program. Ser. A (2014)147: 499-518).
\end{abstract}


\section{Introduction}

Graphs considered in this paper are finite, simple and connected.
Let $G$ be a graph. The vertex set and edge set of $G$ is denoted
by $V(G)$ and $E(G)$, respectively. Let $M$ be a subset of $E(G)$.
The subgraph of $G$ induced by $M$ is denoted by $G[M]$. Write
$$V(M)=\{v\in V(G): \mbox{ there is an } x\in V(G) \mbox{ such that
} vx\in M\}.$$ If $M=\{e\}$, we simply write $V(e)$ instead of
$V(\{e\})$. For a subset $X$ of $V(G)$, $M$ {\it covers} $X$ if $X
\subseteq V(M)$.  $M$ is  a {\it matching} of $G$ if,  for every two
distinct edges $e$ and $f$ in $M$, $V(e)\cap V(f)=\emptyset$. A
matching $M$ of $G$ is {\it perfect} if $V(M)=V(G)$. If $M$ is the union of $k$ matchings of $G$ and $M$ covers $V(G)$, then $M$ is  a {\it
$k$-matching cover} of $G$.  The {\it matching cover number} of $G$,
denoted by $mc(G)$, is the minimum number $k$ such that $G$ has a
$k$-matching cover. A $k$-matching cover is {\it optimal} if
$k=mc(G)$. In \cite{WSY2014}, Wang etc. present an algorithm, called Matching Cover
Algorithm, which finds an optimal matching cover of a graph $G$ in
$O(|V(G)|\cdot|E(G)|)$ time. But this algorithm is not correct. This paper gives a minor revised version of Matching Cover Algorithm, which has the same idea of the original algorithm.

The paper is organized as follows. In Section 2, we present  some
basic results. Section 3 is devoted to Matching Cover Algorithm.

\section{Preliminaries}

We begin with some notions and notations. Let $G$ be a graph. We use
$\Delta(G)$ to denote the maximum degree of $G$. If $G$ is a
bipartite graph with bipartition $(X,Y)$, then the graph $G$ is
denoted by $G[X,Y]$. A {\it star} is a complete bipartite graph
$G[X,Y]$ with $|X|=1$ or $|Y|=1$. If $|X|=1$, we call the only
vertex  in $X$ the {\it center} of the star, and the vertices in $Y$
its {\it ends}. If a star is $K_2$, we may refer to an arbitrary
vertex of $K_2$ as its center and the other one as its end. For a
subset $M$ of $E(G)$ so that $G[M]$ is the union of stars, the
centers of the stars in $G[M]$ are  called the {\it centers} of $M$,
and the ends of the stars in $G[M]$ are called the {\it ends} of
$M$. Note that, if $M$ is a minimal matching cover of $G$, then each
component of $G[M]$ is a star.

Let $u$ be a vertex of $G$. We use $d_G(u)$ to denote the degree
of $u$ in $G$.   $u$ is an {\it isolated vertex} if $d_G(u)=0$.
For a subset $S$ of $V(G)$, we  use $G[S]$ to denote the subgraph of $G$ induced by $S$, $G-S$
the graph obtained from $G$ by deleting all vertices in $S$, and
$N_G(S)$ the neighbor set of $S$, that is,
$$N_G(S)=\{v\in V(G)\setminus S: uv\in E(G) \mbox{ for some } u\in S \}.$$
If $S=\{v\}$, we simply write $G-v$ and $N_G(v)$ instead of
$G-\{v\}$ and $N_G(\{v\})$, respectively.


Let $M$ be a matching of $G$. An {\it $M$-augmenting path} in $G$ is
defined to be a path whose edges are alternately in $M$ and $E(G)
\setminus M$, and whose two ends are not covered by $M$. If the
graph $G$ contains an $M$-augmenting path $P$, then $M\Delta
E(P)$, the symmetric difference of $M$ and $E(P)$,  is a matching of $G$ with $|M\Delta E(P)|=|M| + 1$ and $V(M)\subset
V(M\Delta E(P))$. Berge \cite{Berge57, Bondy08, Lovasz86} in fact proved the
following.

\begin{lem}
 \label{lem:M-augmenting path} \
A matching $M$ of a graph $G$ is a maximum matching if and only if
$G$ contains no $M$-augmenting path.
\end{lem}

Lemma \ref{lem:M-augmenting path} implies that for every matching
$M$ of $G$, there is  a maximum matching of $G$ which covers
$V(M)$. Edmonds established an algorithm for
finding a maximum matching of a graph $G$ in $O(|V(G)|^3)$ time
\cite{Edmonds65, Korte08, Lovasz86}, which is called Edmonds'
Cardinality Matching Algorithm in \cite{Korte08}. As a by-product,
the algorithm provides a constructive proof of the Gallai-Edmonds
Structure Theorem \cite{Korte08, Lovasz86}. To present this
theorem, we first define
$$\begin{array}{l}
D(G) = \{x \in V(G): \mbox{~some maximum matching of $G$
does not cover $x$}\},\\[0.2cm]
A(G)=N_G(D(G)), \mbox{ and }\\[0.2cm]
C(G)=V(G)\setminus (A(G)\cup D(G)).\end{array}$$

\begin{lem}\label{lem:G-E-th} (Gallai-Edmonds
Structure Theorem) \  For a  graph $G$,  the subgraph $G[C(G)]$ has
a perfect matching. Furthermore, if $D(G)\neq
\emptyset$, then\\
(i) each component of $G[D(G)]$ is factor-critical,
and\\
(ii) every maximum matching of $G$ contains a near-perfect matching
of each component\\
\hspace*{0.5cm} of $G[D(G)]$, a perfect matching of $G[C(G)]$ and a
matching which matches all vertices\\
\hspace*{0.5cm}  of $A(G)$ with vertices in distinct components of
$G[D(G)]$.
\end{lem}

\noindent Here, a graph $H$ is {\it factor-critical} if $H-v$ has a
perfect matching for each $v\in V(H)$,  and a matching of $H$ is
{\it near-perfect} if it covers all but one vertex in $H$.

In fact, when Edmonds' Cardinality Matching Algorithm terminates,
besides a maximum matching of $G$, $D(G)$, $A(G)$, and $C(G)$ are
also generated (see \cite{Korte08}). If $D(G)=\emptyset$, then
$G=G[C(G)]$. In this case, $G$ has a perfect matching,  and so
$mc(G)=1$. If $D(G)\neq \emptyset$ and $A(G)=\emptyset$, since $G$
is connected, we have $G=G[D(G)]$, which is factor-critical. Then,
for a vertex $v\in V(G)$, $G-v$ has a perfect matching $M_1$. The
union of $M_1$ and an edge incident with $v$ is an optimal
matching cover of $G$, and so $mc(G)=2$.

When $A(G)\neq\emptyset$, which implies that $D(G)\neq\emptyset$,
we use $G^*$ to denote the simple graph obtained from $G$ by deleting the nontrivial  components of  $G[D(G)]$, the vertices of $C(G)$ and the edges spanned by $A(G)$. Let $D^*$ be the set of isolated vertices in $G[D(G)]$. Then $G^*$ is a bipartite graph with bipartition $(A(G),D^*)$.
The vertices of $G^*$ in $A(G)$ are called {\it $A$-vertices}, and
the  vertices of $G^*$ in $D^*$ are called {\it $D$-vertices}.   Note that $G^*$ is a subgraph of $G$, and if $G^*$ has isolated vertices, then they lie in $A(G)$.
If  $M^*$ is the union of $k$
matchings of $G^*$ and  $M^*$ covers $D^*$, we call   $M^*$  a {\it
$k$-matching $D^*$-cover} of $G^*$.  The {\it matching $D^*$-cover number} of $G^*$,
denoted by $md(G^*)$, is the minimum number $k$ such that $G^*$ has a
$k$-matching $D^*$-cover. A $k$-matching $D^*$-cover is {\it optimal} if
$k=md(G^*)$. Clearly, a matching $D^*$-cover  $M^*$ of $G^*$ is also a matching cover of $G^*[ M^*]$, which covers all but the isolated vertices of $G^*$.

\begin{lem}\label{lem:mcG*}
For each graph $G$ with $A(G)\neq\emptyset$,  if $md(G^*)\leq 1$, then $mc(G)= 2$; if $md(G^*)\geq 2$, then $mc(G)= md(G^*)$.
\end{lem}
\begin{proof}
Since  $A(G)\neq \emptyset$, by Lemma \ref{lem:G-E-th}, $G$ has no perfect
matching. This implies that $mc(G)\geq 2$. Suppose that
$md(G^*)=k$, and suppose that $\cup_{i=0}^k N_i$ is an optimal
matching $D^*$-cover of $G^*$,  where $N_0=\emptyset$, and  each
other $N_i$ is a matching of $G^*$ (and so a matching of $G$).
When $k=0$, that is $D^*=\emptyset$, let $N_1'$ be an arbitrary maximum matching of $G$.
When $k\geq 1$,  by Lemma \ref{lem:M-augmenting path}, let $N_1'$ be  a maximum matching of
$G$ such that $V(N_1)\subseteq V(N_1')$. By Lemma \ref{lem:G-E-th}, $N_1'$ covers $A(G)\cup C(G)$ and for each component of $G[D(G)]$, there is at most one vertex not covered by $N_1'$. Let $U$ be the set of
vertices not covered by $N_1'\cup(\cup_{i=2}^k N_i)$.  Then $U\subseteq D(G)\setminus D^*$. Write $$N'=\{e_u:
u\in U,   \mbox{$e_u$ is an edge  in $G[D(G)]$ which covers
$u$}\}.$$ Then any two edges in $N'$ lie in different components of
$G[D(G)]$. Thus $N'$ is a matching of $G$. When $k\leq 1$,
$N_1'\cup N'$ is a 2-matching cover of $G$, and so $mc(G)=2$.
When $k\geq 2$, write $N_2'=N_2\cup N'$.
Since $V(N_2)\cap V(N')=\emptyset$, $N_2'$ is a matching of $G$. Thus $N_1'\cup N_2' \cup (\cup_{i=3}^k N_i)$ is  a $k$-matching cover of $G$, and so $mc(G)\leq k= md(G^*)$.

To obtain the opposite bound,  let $\cup_{i=1}^l M_i$ be an
optimal matching cover of $G$, where $l=mc(G)$ and each $M_i$ is a
matching of $G$. By Lemma \ref{lem:M-augmenting path} again,   for
each $i$,  $1\leq i\leq l$, there exists a maximum matching $M_i'$
of $G$ such that $V(M_i)\subseteq V(M_i')$.  This implies that
$\cup_{i=1}^l M'_i$ is also a matching cover of $G$. By Lemma
\ref{lem:G-E-th},  when restrict $M'_i$ to $G^*$, we get a matching
$M_i''$ of $G^*$. Then $\cup_{i=1}^l M_i''$ is a matching $D^*$-cover of
$G^*$, and so $mc(G)= l\geq md(G^*)$. The lemma follows.
\end{proof}

The proof of Lemma \ref{lem:mcG*} implies the following.

\begin{cor}\label{cor:mcG*}
Let $G$ be a graph with $A(G)\neq \emptyset$ and $md(G^*)=k$, 
and let $\cup_{i=0}^k N_i$  be an optimal matching $D^*$-cover of $G^*$, where $N_0=\emptyset$.

 (i) If $k=0$, then $G$ has an optimal matching cover which consists of an arbitrary maximum matching of $G$ and a matching of $G[D(G)]$. 

 (ii) If $k=1$, then $G$ has an optimal matching cover which consists of a maximum matching of $G$ covering $V(N_1)$ and  a matching of $G[D(G)]$. 

 (iii) If $k\geq 2$, then $G$ has an optimal matching cover $\cup_{i=1}^k M_i$ such that $M_1$ is a maximum matching of $G$ covering $V(N_1)$, $M_2$ is a matching of $G$ which consists of $N_2$ and a matching of $G[D(G)]$, and $M_i = N_i$ for $3\leq i\leq k$. 
\end{cor}

\section{Matching cover algorithm}

In this section, we are only concerned
with minimal matching covers. It is convenient, therefore, to refer to a minimal matching cover and a  minimal matching $D^*$-cover as a `matching cover' and a ` matching $D^*$-cover', respectively.

We will present an algorithm for finding an optimal matching cover of a graph. We first give an approach to finding an optimal matching $D$-cover of a bipartite graph $G[A,D]$.  Let $M_c$ be a matching $D$-cover of $G$. Then each component of $G[M_c]$ is a star. If $M_c$ is optimal, the maximum star of $G[M_c]$ is minimum among all matching $D$-covers of $G$. Therefore, the
basic idea of our approach is to suitably transfer a matching $D$-cover
$M$ of $G$ into another matching $D$-cover $M'$ of $G$ so that $G[M']$
has a smaller maximum star than $G[M]$.  We formulate this idea in
Lemma \ref{lem:Mc-switching path} after giving the following
definitions.

Let $M_c$ be a matching $D$-cover of $G$ with centers in $A$. For convenience, we call all the vertices in $A$  the centers of $M_c$. It is possible that some centers  are not covered by $M_c$. Use $G[M_c]+A$ to denote the supergraph of $G[M_c]$ obtained by adding the vertices in $A\setminus V(M_c)$ to  $G[M_c]$.    Note that for a center $v$ of $M_c$, if $v$ is covered by $M_c$, then $d_{G[M_c]+A}(v)=d_{G[M_c]}(v)$; otherwise,  $d_{G[M_c]+A}(v)=0$. A center $w$ of $M_c$ is
called a {\it maximum center} if $d_{G[M_c]+A}(w)=\Delta(G[M_c]+A)$. An
{\it $M_c$-alternating path} $P$ in $G$ is defined to be a path
which starts at a center $u$ of $M_c$ and an edge $e\in M_c$ with
$u\in V(e)$, whose vertices are alternately the centers and ends of
$M_c$, and whose edges are alternately in $M_c$ and $E(G)\setminus
M_c$. Let $v$ be the other end of $P$. We denote the path $P$ by
$P_{uv}$. The vertex $v$ might or might not be a center of $M_c$. If
$v$ is a center of $M_c$, $u$ is a maximum center of $M_c$, and
$d_{G[M_c]+A}(u)\geq d_{G[M_c]+A}(v)+2$, then $P_{uv}$ is called an {\it
$M_c$-switching path} and $u$ is called the {\it origin} of
$P_{uv}$. Clearly, the length of an $M_c$-switching path is even. An
example of an $M_c$-switching path $P_{uv}$ is displayed in Fig.1 (b), where $M_c$ is the set of edges depicted by solid lines, and
$P_{uv}$ is induced by the edges depicted by dotted lines.

Suppose that an $M_c$-switching path $P_{uv}$ is given by
$u_1y_1u_2y_2\cdots u_ky_ku_{k+1}$, where $u_1=u$ and $u_{k+1}=v$.
Then each $u_i$ is a center of $M_c$, each $y_i$ is an end of $M_c$,
and for each $i$,  $1\leq i\leq k$, $u_iy_i\in M_c$ and
$y_iu_{i+1}\in E(G)\setminus M_c$. Set $M_c'=(M_c\setminus \{u_1y_1,
u_2y_2, \cdots, u_ky_k\})\cup \{y_1u_2, y_2u_3, \cdots,
y_ku_{k+1}\}$, that is,  $M_c'=M_c \Delta E(P_{uv})$.  Then $M_c'$
is a (minimal) matching $D$-cover of $G$ with the same centers as $M_c$.
Moreover, $d_{G[M_c']}(u)=d_{G[M_c]}(u)-1$,
$d_{G[M_c']+A}(v)=d_{G[M_c]+A}(v)+1$, and $d_{G[M_c']}(w)=d_{G[M_c]}(w)$
for each other center $w$ of $M_c$. We say that $M_c'$ is the {\it
transformation} of $M_c$ with respect to $P_{uv}$ (see Fig. 1(c)).

\setlength{\unitlength}{0.08cm}
\begin{picture}(110,130)(-20,-2)
\multiput(20,120)(30,0){2}{\circle*{1.5}}
\multiput(75,120)(20,0){2}{\circle*{1.5}}
\multiput(110,120)(20,0){1}{\circle*{1.5}}
\multiput(10,100)(10,0){11}{\circle*{1.5}}
\multiput(20,120)(30,0){2}{\line(0,-1){20}}
\multiput(20,120)(30,0){2}{\line(1,-2){10}}
\multiput(20,120)(30,0){2}{\line(-1,-2){10}}
\multiput(75,120)(20,0){2}{\line(1,-4){5}}
\multiput(75,120)(20,0){2}{\line(-1,-4){5}}
\multiput(110,120)(30,0){1}{\line(0,-1){20}}
\multiput(20,120)(30,0){2}{\line(1,-2){10}}
\put(15,120){\makebox(1,0.5)[l]{\small $u$}}
\put(113,120){\makebox(1,0.5)[l]{\small $v$}}
\put(59,92){\makebox(1,0.5)[l]{\small (a)}}

\multiput(20,80)(30,0){2}{\circle{3}}
\multiput(75,80)(20,0){1}{\circle{3}}
\multiput(110,80)(20,0){1}{\circle{3}}
\multiput(20,80)(30,0){2}{\circle*{1.5}}
\multiput(75,80)(20,0){2}{\circle*{1.5}}
\multiput(110,80)(20,0){1}{\circle*{1.5}}
\multiput(10,60)(10,0){11}{\circle*{1.5}}
\multiput(20,80)(30,0){2}{\line(0,-1){20}}
\multiput(20,80)(30,0){2}{\line(1,-2){10}}
\multiput(20,80)(30,0){2}{\line(-1,-2){10}}
\multiput(75,80)(20,0){2}{\line(1,-4){5}}
\multiput(75,80)(20,0){2}{\line(-1,-4){5}}
\multiput(110,80)(30,0){1}{\line(0,-1){20}}
\multiput(20,80)(30,0){2}{\line(1,-2){10}}
\multiput(20.8,80)(0.6,-1.2){17}{\circle*{0.6}}
\multiput(30,60)(1,1){20}{\circle*{0.6}}
\multiput(50.5,80)(0,-1.2){16}{\circle*{0.6}}
\multiput(50,60)(1.25,1){20}{\circle*{0.6}}
\multiput(75.8,80)(0.3,-1.2){17}{\circle*{0.6}}
\multiput(80,60)(1.5,1){20}{\circle*{0.6}}
\put(15,80){\makebox(1,0.5)[l]{\small $u$}}
\put(113,80){\makebox(1,0.5)[l]{\small $v$}}
\put(26,73){\makebox(1,0.5)[l]{\small $e$}}
\put(33,72){\makebox(1,0.5)[l]{\small $P_{uv}$}}
\put(59,52){\makebox(1,0.5)[l]{\small (b)}}

\multiput(20,40)(30,0){2}{\circle{3}}
\multiput(75,40)(20,0){1}{\circle{3}}
\multiput(110,40)(20,0){1}{\circle{3}}
\multiput(20,40)(30,0){2}{\circle*{1.5}}
\multiput(75,40)(20,0){2}{\circle*{1.5}}
\multiput(110,40)(20,0){1}{\circle*{1.5}}
\multiput(10,20)(10,0){11}{\circle*{1.5}}
\multiput(20,40)(30,0){1}{\line(0,-1){20}}
\multiput(20,40)(30,0){2}{\line(-1,-2){10}}
\multiput(95,40)(20,0){1}{\line(1,-4){5}}
\multiput(75,40)(20,0){2}{\line(-1,-4){5}}
\multiput(110,40)(30,0){1}{\line(0,-1){20}}
\multiput(50,40)(30,0){1}{\line(1,-2){10}}
\multiput(30,20)(30,0){1}{\line(1,1){20}}
\multiput(50,20)(30,0){1}{\line(5,4){25}}
\multiput(80,20)(30,0){1}{\line(3,2){30}}
\put(15,40){\makebox(1,0.5)[l]{\small $u$}}
\put(113,40){\makebox(1,0.5)[l]{\small $v$}}
\put(59,12){\makebox(1,0.5)[l]{\small  (c)}}
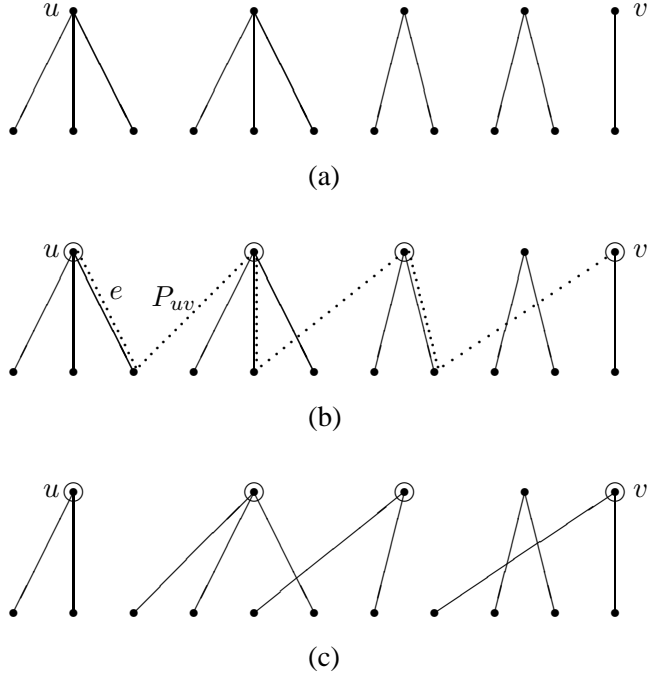
\put(-25,4){\makebox(1,0.5)[l]{\small Fig. 1. (a)  $M_c$, (b) an
$M_c$-switching path $P_{uv}$, (c) the transformation of $M_c$ with
respect to $P_{uv}$.}}
\end{picture}

By Lemma  \ref{lem:mcG*} and Corollary \ref{cor:mcG*}, we may
restrict our attention to $G^*$ (as defined in Section 2), a
bipartite subgraph of a general graph $G$ with bipartition $(A(G), D^*)$, where $D^*$ is the set of isolated vertices in $G[D(G)]$. Recall that the vertices of $G^*$ in $A(G)$ are $A$-vertices and the other vertices of $G^*$ are $D$-vertices. For a subgraph $H$ of $G^*$, let $A_H$ and $D_H$ denote the sets of $A$-vertices and $D$-vertices in $H$, respectively. Then $(A_H,D_H)$ is a bipartition of $H$.

\begin{lem}\label{lem:Mc-switching path}
Let $M_c$ be a matching $D^*$-cover of $G^*$ with $A$-vertices being centers  and with as few maximum centers as possible.  Then
$M_c$ is an optimal matching $D^*$-cover of $G^*$ if and only if $G^*$
contains no $M_c$-switching path.
\end{lem}
\begin{proof}
Suppose that $P_{uv}$ is an $M_c$-switching path in $G^*$. Then $u$
is a maximum center of $M_c$ and $d_{G^*[M_c]+A}(u)\geq
d_{G^*[M_c]+A}(v)+2$. Let $M_c'$ be the transformation of $M_c$ with
respect to $P_{uv}$. Then $M_c'$ is also a matching $D^*$-cover of $G^*$
with $A$-vertices being centers. If $M_c$ has at least two maximum
centers, then $\Delta (G^*[M_c])=\Delta (G^*[M_c'])$ and $M_c'$ has
fewer maximum centers than $M_c$, a contradiction. Thus, $u$ is the
only maximum center of $M_c$, and so, $\Delta (G^*[M_c])>\Delta
(G^*[M_c']).$ This implies that $M_c$ is not an optimal matching
cover of $G^*$.

Conversely, suppose that $M_c$ is not an optimal matching $D^*$-cover of
$G^*$, and let $\tilde{M_c}$ be an optimal matching $D^*$-cover of $G^*$.
We may suppose that $G[\tilde{M_c}]$
is the union of stars with $A$-vertices being centers, that is, in $G[\tilde{M_c}]$ all the $D$-vertices have degree one. Write
$\tilde{k}=md(G^*)$ and $k=md(G^*[M_c])$. Then
\begin{equation}
\Delta (G^*[M_c])=k > \tilde{k}=mc(G^*[\tilde{M_c}])=\Delta
(G^*[\tilde{M_c}]).
\end{equation}
Let $u$ be a maximum center of $M_c$, so that $d_{G^*[M_c]}(u)=k$.
Write $H=G^*[M_c\cup \tilde{M_c}]$. Denote by $U$ the set of all
centers of $M_c$ which are reachable from $u$ by $M_c$-alternating
paths in $H$. Then in $H$ there is no vertex in $A_{G^*}\setminus
U$ adjacent to vertices in $N_{G^*[M_c]}(U)$. So $\tilde{M_c}$ has
no edge joining a vertex in $A_{G^*}\setminus U$ to a vertex in
$N_{G^*[M_c]}(U)$. Since  $\tilde{M_c}$ is a matching $D^*$-cover of
$G^*$ with centers in $A(G)$, we have $N_{G^*[M_c]}(U)\subseteq
N_{G^*[\tilde{M_c}]}(U)$, and so,  $|N_{G^*[M_c]}(U)|\leq
|N_{G^*[\tilde{M_c}]}(U)|$. Let  $v\in U$ be such that
$d_{G^*[M_c]+A}(v)$ is as small as possible. Then
$d_{G^*[M_c]+A}(v)\leq \tilde{k}-1$, since otherwise, by (1), we
have
$$|N_{G^*[M_c]}(U)|\geq k+(|U|-1)\tilde{k}>|U|\tilde{k}\geq
|N_{G^*[\tilde{M_c}]}(U)|,$$ a contradiction. So
$d_{G^*[M_c]+A}(v)\leq \tilde{k}-1\leq k-2=d_{G^*[M_c]}(u)-2$. It
follows that the graph $H$ has an $M_c$-switching path connecting
$u$ and $v$, and so does $G^*$.
\end{proof}

Lemma \ref{lem:Mc-switching path} suggests a natural approach to
finding an optimal matching $D^*$-cover of $G^*$. We start with an
arbitrary matching $D^*$-cover $M_c$ of $G^*$ with $A$-vertices being centers, and search for an $M_c$-switching path. If such an
$M_c$-switching path is found, then the transformation of $M_c$
with respect to the path either has fewer maximum centers or else
has smaller maximum centers. Continuing in this way until we
obtain a matching $D^*$-cover $\tilde{M}_c$ so that  no
$\tilde{M}_c$-switching path exists. Then the final matching $D^*$-cover
$\tilde{M}_c$ is an optimal matching $D^*$-cover of $G^*$.

When generating a switching path, we need the notion of alternating
tree, similar to that in Hungarian Algorithm for finding a maximum
matching in a bipartite graph. Let $M_c$ be a matching $D^*$-cover of
$G^*$, and $u$ a center of $M_c$. A tree $T$ which is a subgraph of
$G^*$ is an {\it $M_c$-alternating tree} rooted at $u$ if $u\in
V(T)$ and, for each $v\in V(T)\setminus\{u\}$, the unique path in
$T$ starting at $u$ and ending at $v$ is an $M_c$-alternating path.

Let $u$ be a maximum center of $M_c$. By means of a simple
tree-search algorithm, we can find either an $M_c$-switching path
starting at $u$ or a maximal $M_c$-alternating tree rooted at $u$.
We begin with a trivial $M_c$-alternating tree rooted at $u$ which
consists of just one component of $G^*[M_c]$, a star with
center $u$. At each stage, we attempt to extend the current
$M_c$-alternating tree to a larger one. Consider an
$M_c$-alternating tree $T$ rooted at $u$ which has no
$M_c$-switching path. Recall that $A_T$ consists of some centers
of $M_c$, $D_T$ consists of some ends of $M_c$, and $(A_T, D_T)$
is a bipartition of $T$ with $u\in A_T$. If there is an edge
$xy\in E(G^*)$ with $x\in D_T$ and $y\in A_{G^*}\setminus A_T$,
then we grow $T$ into a larger $M_c$-alternating tree by adding
the edge $xy$ and $S_y$, where $S_y$ is a component of
$G^*[M_c]+A$ containing $y$ (see Fig. 2(b)). In the case that
$d_{G^*[M_c]+A}(y)\leq d_{G^*[M_c]+A}(u)-2$, we find an
$M_c$-switching path (see Fig. 2(c)). If  no such edge $xy\in
E(G^*)$ with $x\in D_T$ and $y\in A_{G^*}\setminus A_T$ exists,
$T$ is  a maximal $M_c$-alternating tree rooted at $u$.

\setlength{\unitlength}{0.08cm}
\begin{picture}(110,73)(-5,-6)
\multiput(30,60)(30,0){1}{\circle*{1.4}}
\multiput(18,50)(8,0){4}{\circle*{1.4}}
\multiput(18,40)(24,0){2}{\circle*{1.4}}
\multiput(8,30)(10,0){3}{\circle*{1.4}}
\multiput(37,30)(10,0){2}{\circle*{1.4}}
\multiput(30,60)(30,0){1}{\line(-6,-5){12}}
\multiput(30,60)(30,0){1}{\line(6,-5){12}}
\multiput(30,60)(30,0){1}{\line(2,-5){4}}
\multiput(30,60)(30,0){1}{\line(-2,-5){4}}
\multiput(18,50)(0,-2.6){4}{\line(0,-1){1.5}}
\multiput(42,50)(0,-2.6){4}{\line(0,-1){1.5}}
\multiput(18,40)(30,0){1}{\line(-1,-1){10}}
\multiput(18,40)(30,0){1}{\line(0,-1){10}}
\multiput(18,40)(30,0){1}{\line(1,-1){10}}
\multiput(42,40)(10,0){1}{\line(0,-1){10}}%
\multiput(42,30)(24,0){1}{\circle*{1.4}}%
\multiput(42,40)(10,0){1}{\line(1,-2){5}}
\multiput(42,40)(10,0){1}{\line(-1,-2){5}}
\put(27,62){\makebox(1,0.5)[l]{\small $u$}}

\multiput(80,60)(30,0){1}{\circle*{1.4}}
\multiput(68,50)(8,0){4}{\circle*{1.4}}
\multiput(68,40)(24,0){2}{\circle*{1.4}}
\multiput(58,30)(10,0){3}{\circle*{1.4}}
\multiput(87,30)(10,0){2}{\circle*{1.4}}
\multiput(92,40)(10,0){1}{\line(0,-1){10}}%
\multiput(92,30)(24,0){1}{\circle*{1.4}}%
\multiput(78,20)(10,0){1}{\circle*{1.4}}
\multiput(73,10)(10,0){2}{\circle*{1.4}}
\multiput(78,20)(30,0){1}{\line(-1,-2){5}}
\multiput(78,20)(30,0){1}{\line(1,-2){5}}
\multiput(80,60)(30,0){1}{\line(-6,-5){12}}
\multiput(80,60)(30,0){1}{\line(6,-5){12}}
\multiput(80,60)(30,0){1}{\line(2,-5){4}}
\multiput(80,60)(30,0){1}{\line(-2,-5){4}}
\multiput(68,50)(0,-2.6){4}{\line(0,-1){1.5}}
\multiput(92,50)(0,-2.6){4}{\line(0,-1){1.5}}
\multiput(68,40)(30,0){1}{\line(-1,-1){10}}
\multiput(68,40)(30,0){1}{\line(0,-1){10}}
\multiput(68,40)(30,0){1}{\line(1,-1){10}}
\multiput(92,40)(10,0){1}{\line(1,-2){5}}
\multiput(92,40)(10,0){1}{\line(-1,-2){5}}
\multiput(78,30)(0,-2.6){4}{\line(0,-1){1.5}}
\put(77,62){\makebox(1,0.5)[l]{\small $u$}}
\put(80,31){\makebox(1,0.5)[l]{\small $x$}}
\put(80,21){\makebox(1,0.5)[l]{\small $y$}}

\multiput(130,60)(30,0){1}{\circle*{1.4}}
\multiput(118,50)(8,0){4}{\circle*{1.4}}
\multiput(118,40)(24,0){2}{\circle*{1.4}}
\multiput(108,30)(10,0){3}{\circle*{1.4}}
\multiput(137,30)(10,0){2}{\circle*{1.4}}
\multiput(142,40)(10,0){1}{\line(0,-1){10}}%
\multiput(142,30)(24,0){1}{\circle*{1.4}}%
\multiput(128,20)(10,0){1}{\circle*{1.4}}
\multiput(123,10)(10,0){2}{\circle*{1.4}}
\multiput(130,60)(30,0){1}{\line(-6,-5){12}}
\multiput(130,60)(30,0){1}{\line(6,-5){12}}
\multiput(130,60)(30,0){1}{\line(2,-5){4}}
\multiput(130,60)(30,0){1}{\line(-2,-5){4}}
\multiput(118,50)(0,-2.6){4}{\line(0,-1){1.5}}
\multiput(142,50)(0,-2.6){4}{\line(0,-1){1.5}}
\multiput(118,40)(30,0){1}{\line(-1,-1){10}}
\multiput(118,40)(30,0){1}{\line(0,-1){10}}
\multiput(118,40)(30,0){1}{\line(1,-1){10}}
\multiput(142,40)(10,0){1}{\line(1,-2){5}}
\multiput(142,40)(10,0){1}{\line(-1,-2){5}}
\multiput(128,30)(0,-2.6){4}{\line(0,-1){1.5}}
\multiput(128,20)(30,0){1}{\line(-1,-2){5}}
\multiput(128,20)(30,0){1}{\line(1,-2){5}}
\multiput(129,60)(-1.2,-1){10}{\circle*{0.6}}
\multiput(117.5,50)(0,-1.2){9}{\circle*{0.6}}
\multiput(117.2,40)(1,-1){10}{\circle*{0.6}}
\multiput(127.5,30)(0,-1.2){9}{\circle*{0.6}}
\put(127,62){\makebox(1,0.5)[l]{\small $u$}}
\put(130,21){\makebox(1,0.5)[l]{\small $v$}}
\put(27,3){\makebox(1,0.5)[l]{\small (a) }}
\put(77,3){\makebox(1,0.5)[l]{\small (b) }}
\put(127,3){\makebox(1,0.5)[l]{\small (c) }}
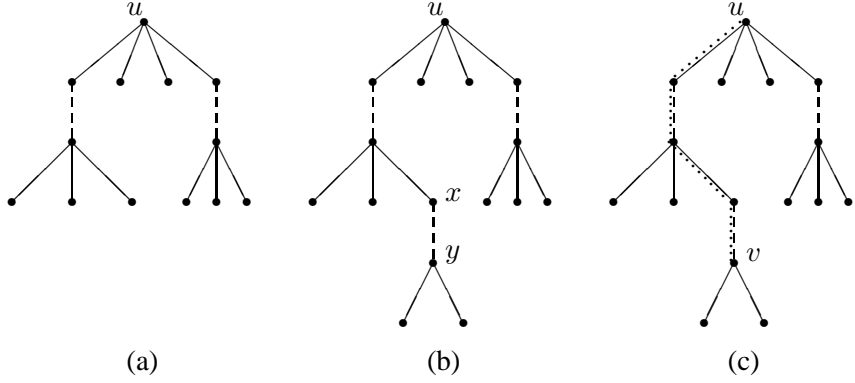
\put(-5,-5){\makebox(1,0.5)[l]{\small Fig. 2. (a) An
$M_c$-alternating tree, (b) a larger $M_c$-alternating tree, (c) an
$M_c$-switching path.}}
\end{picture}

In order to decrease the iteration times of the algorithm, we
first construct a near-maximal $M_c$-alternating forest. Here a
{\it near-maximal $M_c$-alternating forest} $F$ is a union of
$M_c$-alternating trees each of which roots at a maximum center of
$M_c$, say $F=\cup_{i=1}^t T_i$, such that all maximum centers of
$M_c$ are in $F$ and for each $i$, $1\leq i\leq t$, $T_i$ is
maximal in $G^*-V(\cup_{j=0}^{i-1} T_j)$ ($T_0$ is null). We then
find an $M_c$-switching path $P_{uv}$ in $F$ (if exists) such that
$d_{G^*[M_c]+A}(v)$ is as small as possible. Using such switching
path to update matching $D^*$-cover $M_c$, the vertex $v$ may become the
origin of a switching path at most once.

We are now ready for a detailed description of the algorithm for
finding an optimal matching cover of a graph, which is referred to
as {\it Matching Cover Algorithm}.

\begin{algo} {\bf(Matching Cover Algorithm)}

\noindent Input: A connected graph $G$.

\noindent Output: An optimal matching cover $\tilde{M_c}$ of $G$.

\begin{enumerate}[Step 1:]
\item \label{step:1} Use Edmonds' Cardinality Matching Algorithm to find a
maximum matching $M$ of $G$, and find $D(G), A(G)$, and $C(G)$.

Construct the bipartite graph $G^*$.

Set $M':=M\cap E(G[C(G)])$, and $M^\star:=M\cap E(G^*)$.

\item \label{step:0} If $A(G)=\emptyset$,
when $D(G)=\emptyset$, set $\tilde{M_c}:=M$; when
$D(G)\neq\emptyset$, set $\tilde{M_c}:=M\cup\{e\}$, where $e$ is an
edge of $G$ incident with the vertex not covered  by $M$.

Go to Step \ref{step:10}.

\item \label{step:2} For each $D$-vertex of $G^*$ not covered by $M^\star$, choose
exactly one edge of $G^*$ which  joins this vertex to an $A$-vertex,
and let $M_0$ be the set of these edges.

Set $M_c:=M^\star\cup M_0$.

\item\label{step:3}  Let $U$ be the set of maximum centers of $M_c$ (We may suppose that $U\subseteq A_{G^*}$).

Set $F:=\emptyset$.

If the vertices in $U$ have degree at most one in $G^*[M_c]+A$, then go to Step \ref{step:9}.

\item\label{step:4}  If~ $U=\emptyset$, then go to Step
\ref{step:8}.

\item\label{step:5}  Let $u$ be a vertex in $U$.

Let $T_u$ be the component of $G^*[M_c]$ containing $u$.

\item\label{step:6}  Let $xy$ be an edge of $G^*$ with $x\in D_{T_u}$ and $y\in
A_{G^*}\setminus (A_{T_u}\cup A_F)$.

If no such edge exists, set $F:=F\cup T_u$ ($u$ is the root of
$T_u$) and $U:=U\setminus A_F$, and go to Step \ref{step:4}.

\item\label{step:7}  Set $T_u:=(T_u\cup S_y)+xy$, where $S_y$ is the component of $G^*[M_c]+A$ containing $y$, and then go to Step \ref{step:6}.

\item \label{step:8} Let $v$ be a vertex in $A_F$ such that $d_{G^*[M_c]+A}(v)$ as
small as possible.

Let $u$ be the root of the component of $F$ containing $v$.

If $d_{G^*[M_c]+A}(v)\geq d_{G^*[M_c]+A}(u)-1$, then go to Step
\ref{step:9}.

If $d_{G^*[M_c]+A}(v)\leq d_{G^*[M_c]+A}(u)-2$, do:

Let $P_{uv}$ be the path in $F$ starting at $u$ and ending at $v$.

Replace $M_c$ by the transformation of $M_c$ with respect to
$P_{uv}$.

Go to Step \ref{step:3}.

\item \label{step:9}  From a maximal matching $N$ in $G^*[M_c]$, use Edmonds' Cardinality Matching Algorithm to find a maximum matching $\tilde{M}$ of $G-V( M')$.

    For each vertex in $G[D(G)]$ which is not covered by $\tilde{M}\cup M_c$, choose  an edge in $G[D(G)]$ which covers this vertex. Let $M''$ be the set of these edges.

Set $\tilde{M_c}:=M'\cup \tilde{M}\cup M''\cup (M_c\setminus N)$.

\item \label{step:10} Return $\tilde{M_c}$.
\end{enumerate}
\end{algo}

\begin{thm}\label{thm:algorithm time}
Matching Cover Algorithm correctly determines an optimal matching
cover of a graph $G$ in $O(n^3)$ time, where $n= |V(G)|$.
\end{thm}
\begin{proof}
We first consider the case when $A(G)=\emptyset$ (Step
\ref{step:0}). If $D(G)=\emptyset$, then $G=G[C(G)]$ and so $M$ is a
perfect matching of $G$, which is an optimal matching cover of $G$.
If $D(G)\neq\emptyset$, we have $G=G[D(G)]$, which is factor-critical. So $M$ is a near-perfect matching of $G$. This implies
that the union of $M$ and an edge $e$ of $G$ incident with the
vertex not covered by $M$ is an optimal matching cover of $G$.

We next consider the case when $A(G)\neq\emptyset$.  By Lemma
\ref{lem:G-E-th}, $G$ has no perfect matchings and  $M'$ and $M^\star$ obtained in Step \ref{step:1} are a perfect matching of $G[C(G)]$ and  a matching of $G^*$, respectively. Note that $M_0$
obtained in Step \ref{step:2} covers the $D$-vertices of  $G^*$ which are not covered by $M^\star$. So $M^\star\cup M_0$ is a matching $D^*$-cover of
$G^*$ with its centers being $A$-vertices. The first iteration
of Matching Cover Algorithm  begins with the matching $D^*$-cover
$M_c=M^\star\cup M_0$, which is minimal. If $d_{G^*[M_c]+A}(u)\leq 1$ for a vertex $u\in U$ (Step \ref{step:3}), then $M_c\setminus N=\emptyset$, where $N$ is a maximal matching in $G^*[M_c]$ (Step \ref{step:9}). By Lemma \ref{lem:G-E-th}, the union of the perfect matching $M'$  of $G[C(G)]$ and the  maximum matching $\tilde{M}$ of $G-V(M')$ is a maximum matching of $G$, which covers $A(G)$, and for each component of $G[D(G)]$, there is at most one vertex not covered by $M'\cup \tilde{M}$. Therefore $M''$ obtained in Step \ref{step:9} is a matching, which covers $V(G)\setminus V(M'\cup \tilde{M})$. Consequently,  $(M'\cup \tilde{M})\cup M''$, a union of two  matchings of $G$, is an optimal matching cover of $G$.

Now return to the case when $d_{G^*[M_c]+A}(u)\geq 2$ for a vertex $u\in U$ (Step \ref{step:3}). In each iteration of Step \ref{step:6},  since  $xy$ is an edge of $G^*$ such that $x\in D_{T_u}$ and $y\in A_{G^*}\setminus
(A_{T_u}\cup A_F)$, $xy$ is  not in $M_c$, and so the graph $T_u\cup
S_y+xy$ obtained in Step \ref{step:7} is an $M_c$-alternating tree
rooted at a maximum center $u$ of $M_c$. If no such edge $xy$
exists, then $T_u$ is a maximal $M_c$-alternating tree in $G^*-V(F)$
and $F$ is replaced by $F\cup T_u$. When all maximum center of $M_c$
are scanned (from Step \ref{step:4} to Step \ref{step:7}), $F$ is a
near-maximal $M_c$-alternating forest.

In Step \ref{step:8}, for the vertex $v\in A_F$, if
$d_{G^*[M_c]+A}(v)\leq d_{G^*[M_c]+A}(u)-2$, then the path $P_{uv}$ is
an $M_c$-switching path. This implies that the updated $M_c$ at
the end of Step \ref{step:8}, which is the transformation of $M_c$
with respect to $P_{uv}$, is also a matching $D^*$-cover of $G^*$ with
$A$-vertices being centers. In Step \ref{step:9}, since $\tilde{M}$ is a
maximum matching of $G-V( M')$ obtained from the matching $N$, $\tilde{M}$ covers $V(N)$. Recall that $M'$ covers $C(G)$, $\tilde{M}\cup M'$ is a maximum matching of $G$, $\tilde{M}$ covers $A(G)$, $M_c$ covers $D^*$, and the vertices in $G[D(G)]$  not covered by $\tilde{M}\cup M_c$ is covered by the matching $M''$ of  $G[D(G)]$ obtained in Step 10. 
Therefore, when the algorithm
terminates, if  $M_c=N$, then $(M'\cup \tilde{M})\cup M''$ is a matching cover of $G$, which has two matchings $M'\cup \tilde{M}$ and $M''$ ; if  $M_c\supset N$, then the union of $M''$ and a maximal matching of $G[M_c\setminus N]$ is a matching of $G$,  and so
$(M'\cup \tilde{M})\cup M''\cup (M_c\setminus N)$ is a matching cover of $G$, which has the same number of maximal matchings as $M_c$ . By Lemma
\ref{lem:mcG*}, it suffices to show that the final $M_c$ is an
optimal matching $D^*$-cover of $G^*$.

Now consider the final $M_c$ and $F$ in the algorithm. Note that
each component of $G[M_c]$ is a star and  $F$ is a near-maximal
$M_c$-alternating forest. We have the following facts.

\noindent $\bullet$  All the maximum centers of $M_c$ belong to
$A_F$.

\noindent $\bullet$  $F$ is the union of pairwise disjoint
$M_c$-alternating trees each of which roots at a maximum
\\ \hspace*{0.25cm} center of $M_c$.

\noindent $\bullet$  $G^*$ has no edge $xy$ with $x\in D_F$ and
$y\in A_{G^*}\setminus A_F$.

\noindent Hence, if $G^*$ has an $M_c$-switching path $P$, then
$V(P)\subseteq V(F)$. It follows that  if $F$ has no $M_c$-switching
path, so does $G^*$.

We claim that $\Delta(G^*[M_c])- d_{G^*[M_c]+A}(v)\leq 1$ for every
vertex $v\in A_F$. Suppose to the contrary that there is a vertex
$v\in A_F$ such that $\Delta(G^*[M_c])- d_{G^*[M_c]+A}(v)\geq 2$.
Since $v\in A_{T_u}$ for some component $T_u$ of $F$, the path in
$T_u$ connecting $u$ and $v$ is an $M_c$-switching path, and then
$M_c$ should be updated in Step \ref{step:8} of the algorithm, a
contradiction. The claim follows.

From the above claim, we see that $G^*$ has no $M_c$-switching
paths. By Lemma \ref{lem:Mc-switching path}, $M_c$ is an optimal
matching $D^*$-cover of $G^*$.

We now consider the running time of the algorithm. Using Edmonds'
Cardinality Matching Algorithm in Step \ref{step:1} and  Step
\ref{step:9} takes $O(n^3)$ time. Finding the matching $M''$ in Step \ref{step:9} takes  $O(|E(G^*)|)$ time. Dealing with the case when
$A(G)=\emptyset$ in Step \ref{step:0} takes $O(n)$ time. Finding
the initial matching $D^*$-cover $M_c$ of $G^*$ in Step \ref{step:2}
takes $O(|E(G^*)|)$ time. For each given matching $D^*$-cover $M_c$ (not
including the final one) of $G^*$, it takes $O(|E(G^*)|)$ time to
find a near-maximal $M_c$-alternating forest $F$ and an
$M_c$-switching path of $G^*$ (if exists). To complete the proof,
we need only  show that there are $O(n)$ iterations to update
matching $D^*$-cover $M_c$. Write $k^*= md(G^*)$.

For each iteration other than the last one, consider the current
matching $D^*$-cover $M_c$ of $G^*$ and the near-maximal $M_c$-alternating
forest $F$. Since $M_c$ is updated later, in Step \ref{step:8} there
is an $M_c$-switching path $P_{uv}$ in $F$ such that $v\in A_F$ and the degree of
$v$ in $G^*[M_c]+A$ is minimum. We claim that
$d_{G^*[M_c]+A}(v)< k^*$. To show this, let $M_c'$ be an optimal
matching $D^*$-cover of $G^*$ with $A$-vertices being centers. Recall
that, for the near-maximal $M_c$-alternating forest $F$, $G^*$ has
no edge $xy$ with $x\in D_F$ and $y\in A_{G^*}\setminus A_F$. Then
we have $D_F\subseteq N_{G^*[M_c']}(A_F)$, and so, $|D_F|\leq
|N_{G^*[M_c']}(A_F)|$. Since $d_{G^*[M_c]}(u)\geq
d_{G^*[M_c]+A}(v)+2$, we have $\lceil\frac{|D_F|}{|A_F|}\rceil >
d_{G^*[M_c]+A}(v)$. Thus
$$ d_{G^*[M_c]+A}(v)<\lceil\frac{|D_F|}{|A_F|}\rceil\leq
\lceil\frac{|N_{G^*[M_c']}(A_F)|}{|A_F|}\rceil\leq \Delta
(G^*[M_c'])=k^*.$$ The claim follows.

Note that, when the algorithm terminates, $G^*$ has no
$M_c$-switching path. But an optimal matching $D^*$-cover of $G^*$ may
appear before the termination of the algorithm. This is because,
for an optimal matching $D^*$-cover $M_c$ of $G^*$, there may still
exist an $M_c$-switching path $P_{uv}$ in the near-maximal
$M_c$-alternating forest $F$ such that the transformation of $M_c$
with respect to $P_{uv}$ has fewer maximum centers than $M_c$.

Let $\mathcal{P}$ be the set of switching paths $P_{uv}$ which are
used in the algorithm to update $M_c$ before $M_c$ is optimal
(that is, $\Delta(G^*[M_c])>k^*$). By the above claim, for each
$P_{uv}$ in $\mathcal{P}$, we have $d_{G^*[M_c]}(u)> k^*$ and
$d_{G^*[M_c]+A}(v)< k^*$. Therefore, for the new matching $D^*$-cover
$M_c'$ which is the transformation of $M_c$ with respect to
$P_{uv}$, the degree of $v$ in $G^*[M_c']$ is increased by 1, and
so $d_{G^*[M_c']+A}(v)\leq k^*$. This implies that before the first
optimal matching $D^*$-cover of $G^*$ appears, the vertex $v$ never
becomes the origin of a switching path. Write $V'=\{u: \mbox{ there is a
switching path $P_{uv}$ in } \mathcal{P}\}$. Before the first
optimal matching $D^*$-cover of $G^*$ appears, we always have
$\sum_{u\in V'} d_{G^*[M_c]+A}(u) < |D_{G^*}|$.  So
$|\mathcal{P}|\leq |D_{G^*}|$, which implies that there are at
most $|D_{G^*}|$ iterations to update $M_c$ which is not optimal.

Whence $M_c$ is an optimal matching $D^*$-cover of $G^*$ in the algorithm,
the maximum degree of $G^*[M_c]$ will not change in the later
iterations. Because there are at most $|A_{G^*}|$ maximum centers in
an optimal matching $D^*$-cover of $G^*$, at most $|A_{G^*}|-1$ iterations
are used to update $M_c$ which is optimal. So the total number of
times to update $M_c$ is at most $|V(G^*)|$. The proof is complete.
\end{proof}

\noindent{\bf Remark:} \ For a graph $G$ with $n$ vertices and $m$ edges,
 Goldberg and Karzanov \cite{Golberg95} presented a faster algorithm to
find a maximum matching of $G$ in $O(\sqrt{n}m\log_n
\frac{n^2}{m})$ time. Furthermore, if we have a maximum matching,
the Edmonds-Gallai decomposition can be found in $O(n^2)$ time
(see \cite{Schrijver03}, page 425). Therefore, after replace
Edmonds' Cardinality Matching Algorithm in Algorithm 1 by Goldberg
and Karzanov's matching algorithm, Step \ref{step:1} and Step
\ref{step:9} takes $O(nm)$ time. Recall that Step \ref{step:0} and
Step \ref{step:2} takes $O(m)$ time, and from Step \ref{step:3} to
Step \ref{step:8}, there are $O(n)$ iterations each of which takes
$O(m)$ time. Consequently, we have an algorithm for finding an
optimal matching cover of a graph $G$ with running time $O(nm)$.

\end{document}